\documentclass[preprint,12pt]{elsarticle}
\usepackage{latexsym}
\usepackage{amssymb}
\usepackage{amsmath}
\usepackage{amsthm}
\usepackage{verbatim}

\numberwithin{equation}{section}

\newtheorem{theorem}{Theorem}[section]
\newtheorem{lemma}{Lemma}[section]
\newtheorem{proposition}{Proposition}[section]

\begin{document}

\begin{frontmatter}



\title{Global Stability of Traveling Wave Fronts for a Population Dynamics Model with  Quiescent Stage and Delay
}


\author[ad1]{Yonghui Zhou}
\ead{zhouyh318@nenu.edu.cn}
\author[ad1,ad2]{Shuguan Ji\corref{cor}}
\ead{jishuguan@hotmail.com}
\address[ad1]{School of Mathematics and Statistics and Center for Mathematics and Interdisciplinary Sciences, Northeast Normal University, Changchun 130024, P.R. China}
\address[ad2]{School of Mathematics, Jilin University, Changchun 130012, P.R. China}
\cortext[cor]{Corresponding author.}

\begin{abstract}
This paper is concerned with the globally exponential stability of traveling wave fronts for a class of population dynamics model with quiescent stage and delay. First, we establish the comparison principle of solutions for the population dynamics model. Then, by the weighted energy method combining comparison principle, the globally exponential stability of traveling wave fronts of the population dynamics model under the quasi-monotonicity conditions is established.
\end{abstract}


\end{frontmatter}


\section{Introduction}

In this paper, we investigate the globally exponential stability of traveling wave fronts for the following system
\begin{equation}
\label{101}
\begin{cases}
u_{1t}(t, x)=D[J\ast u_{1}(t, x)-u_{1}(t, x)]+f\big(u_{1}(t, x), u_{1}(t-\tau,x)\big)
\\
\ \ \ \ \ \ \ \ \ \ \ \ \ \
-\gamma_{1}u_{1}(t, x)+\gamma_{2}u_{2}(t, x), \\
u_{2t}(t, x)=\gamma_{1}u_{1}(t, x)-\gamma_{2}u_{2}(t, x)
\end{cases}
\end{equation}
with the initial data
\begin{equation}
\begin{cases}
u_{1}(s, x)=u_{10}(s, x),\ (s, x)\in[-\tau,0]\times\mathbb{R}, \\
u_{2}(0, x)=u_{20}(x),\ x\in\mathbb{R},
\label{102}
\end{cases}
\end{equation}
where $u_{1}(t,x)$ and $u_{2}(t,x)$ denote the densities of mobile and stationary subpopulations, respectively, $\gamma_{1}>0$ is the rate of switching from a mobile state to stationary state and $\gamma_{2}>0$ is the rate of switching from a stationary state to mobile state,
$\tau\geq 0$\ denotes the time delay, $J(y)$ is a continuous nonnegative kernel function satisfying $J(-y)=J(y)$, $\int_{\mathbb{R}}J(y)dy = 1$ and $\int_{\mathbb{R}}e^{-\lambda y}J(y)dy<\infty$ for all $\lambda > 0$, $f$ is the reproduction function that satisfies the following assumptions
\begin{description}
\item[(A1)] $f(0,0)=f(K,K)=0,\ f\in C^{2}([0,K]^{2},\mathbb{R}),\ f(u,u)>0$\ for\ $u\in (0,K),$\ and\ $\partial_{2}f(u,v)\geq 0$\ for\ $(u,v)\in[0,K]^{2}$,\ $\partial_{ij}f(u,v)\leq 0$ ($i, j=1, 2$) for $(u,v)\in[0,K]^{2}$ where\ $K$\ is a positive constant;
\item[(A2)] $\partial_{1}f(0,0)u+\partial_{2}f(0,0)v\geq f(u,v)$\ for\ $(u,v)\in[0,K]^{2},\  \partial_{1}f(K,K)+\partial_{2}f(K,K)<0$.
\end{description}
By (A1)-(A2),
if $\partial_{1}f(K,K)+\partial_{2}f(K,K)<\gamma_{1}-3\gamma_{2}$, then there exists a positive constant $\beta$ such that
\begin{equation}
D\int_{-\infty}^{0}J(y)e^{-\beta y}dy<\frac{D}{2}+\gamma_{1}-\partial_{1}f(K,K)-\partial_{2}f(K,K)-3\gamma_{2}.
\label{104}
\end{equation}
The system \eqref{101} could
represent a model for a population where individuals migrate and reproduce
and are subject to randomly occurring inactive phases. A special case of system \eqref{101} is the following delayed diffusive Nicholson's blowflies equation with a quiescent stage
\begin{eqnarray*}
\begin{cases}
u_{1t}(t, x)=D[J\ast u_{1}(t, x)-u_{1}(t, x)]-d(u_{1}(t, x))+ b(u_{1}(t-\tau,x))e^{-\mu_{0}\tau}\\
\ \ \ \ \ \ \ \ \ \ \ \ \ \ -\gamma_{1}u_{1}(t, x)+\gamma_{2}u_{2}(t, x), \\
u_{2t}(t, x)=\gamma_{1}u_{1}(t, x)-\gamma_{2}u_{2}(t, x),
\end{cases}
\end{eqnarray*}
where $u_{1}(t,x)$ and $u_{2}(t,x)$ are the densities of mobile and stationary subpopulations of the mature blowflies at the time $t$ and point $x$, respectively, $d(u)$ is the death rate function, $b(u)$ is the birth rate function,
$\mu_{0}$ is the death rate of the juvenile, the delay $\tau\geq 0$ is the duration of the juvenile state, $e^{-\mu_{0}\tau}$ is the survival rate of the juvenile(see, e.g. \cite{Bocharov,Hadeler2009} for detailed interpretation).

The traveling waves for \eqref{101} connecting two constant states $u_{-}=(0,0)$ and $u_{+}=\big(K,\frac{\gamma_{1}K}{\gamma_{2}}\big)$  at far fields are the special solutions to \eqref{101} in the form of $u(t,x)=\phi(x+ct)$, namely,
\begin{equation}
\begin{cases}
c\phi_{1}'(\xi)=D[J\ast\phi_{1}(\xi)-\phi_{1}(\xi)]+f\big(\phi_{1}(\xi), \phi_{1}(\xi-c\tau)\big)-\gamma_{1}\phi_{1}(\xi)+\gamma_{2}\phi_{2}(\xi), \\
c\phi_{2}'(\xi)=\gamma_{1}\phi_{1}(\xi)-\gamma_{2}\phi_{2}(\xi)
\end{cases}
\label{105}
\end{equation}
with the asymptotic boundary condition
\begin{equation}
(\phi_{1}(-\infty),\phi_{2}(-\infty))=(0,0)\ and\ (\phi_{1}(+\infty),\phi_{2}(+\infty))=\Big(K,\frac{\gamma_{1}K}{\gamma_{2}}\Big).
\label{106}
\end{equation}
It is obvious that, by (A1)-(A2) and $\phi_{1}(+\infty)=K$,
there exists a sufficiently large number $\xi_{0}$ such that for $\xi\geq \xi_{0}$,
\begin{equation}
\partial_{i}f(\phi_{1}(\xi),\phi_{1}(\xi-c\tau))<\frac{\partial_{i}f(K,K)+\gamma_{2}}{2},\ \ i=1,2.
\label{103}
\end{equation}

Due to the important role in biology, epidemiology and population dynamics, the traveling wave solutions of the reaction diffusion equations were widely studied \cite{Aronson,Bocharov,Chen,Mei6,Fife1977,Guo2018,
Hadeler2009,huang2012,Lin0,Lin1,Mei1,Mei2,Mei4,Mei5,Sattinger,
Schaaf1987,Volpert1990,Volpert1997,Wang1,Wang2,Lv2015,Wang2011,Yang2018,Yu2013,Yu2017,Zhang2018,Zhou2013}. In \cite{Zhou2013}, by using Schauder's fixed point theorem and upper-lower solution method, Zhou et al. established the existence of traveling wave fronts for system \eqref{101}.
Besides, the stability of the traveling wave solutions is an important issue in the traveling wave theory, especially the stability of traveling wave solutions of nonlocal dispersal equations.
For traveling wave fronts, the frequently used methods are squeezing technique \cite{Fife1977,Volpert1990}, the technical weighted energy method \cite{Mei1,Mei2}, the weighted energy method combining comparison principle \cite{Guo2018,Mei4,Mei5,Lv2015,Zhang2018}, the weighted energy method combining Fourier transform \cite{Wang2011} and spectral analysis method \cite{Sattinger,Schaaf1987,Volpert1997}. Among the above methods, the most classic method is the weighted energy method combined with the comparison principle, which was developed by Mei et al. \cite{Mei4,Mei5}. For example, by using the weighted energy method combining comparison principle, Mei et al. \cite{Mei4} obtained the global stability of the traveling wave fronts for time delayed reaction diffusion equation with local nonlinearity; Wang and Lv \cite{Lv2015} continued to use this method to establish stability of traveling wave fronts for nonlocal reaction diffusion equation with delay.

Although the stability of traveling wave fronts for delayed reaction diffusion equations and nonlocal dispersal equation has been studied intensively, there are few
results about the traveling wave fronts for nonlocal dispersal systems, see only \cite{Guo2018,Zhang2018}. So far, however, there is no result on stability of
travelling wave solutions for system \eqref{101}.
Inspired by above works, we investigate the global stability of traveling wave fronts for a class of nonlocal dispersal system \eqref{101} with delay.

\noindent

The rest of this paper is organized as follows. In Section 2, we
introduce some preliminaries and state our main results. In
Section 3, we prove our stability theorem. In the appendix, we prove the key inequalities used in Section 3.

\textbf{Notations.} Throughout this paper, $C>0$ denotes a generic
constant, $C_{i}>0(i=1,2,...)$ represents a specific constant. Let
$I$ be an interval. $L^{2}(I)$ is the space of the square integrable
functions defined on $I$, and $H^{k}(I)(k\geq 0)$ is the Sobolev
space of the $L^{2}-$function $h(x)$ defined on the interval $I$
whose derivatives $\frac{d^{i}}{dx^{i}}h(i=1,2,...,k)$ also belong
to $L^{2}(I).\ L^{2}_{w}(I)$ denotes the weighted $L^{2}-$space
with a weight function $w(x)>0$ and its norm is defined by
$
\|h\|_{L^{2}_{w}}=\left(\int_{I}w(x)\left|h(x)\right|^{2}dx\right)^{\frac{1}{2}},
H^{k}_{w}(I)$ is the weighted Sobolev space with the norm given by
\begin{equation*}
\|h\|_{H^{k}_{w}}=\left(\sum_{i=0}^{k}\int_{I}w(x)\left|\frac{d^{i}}
{dx^{i}}h(x)\right|^{2}dx\right)^{\frac{1}{2}}.
\end{equation*}
Let $T>0$ be a number and $\mathcal {B}$ be a Banach space. $C([0,T];\mathcal {B})$ is the space of $\mathcal
{B}-$valued continuous functions on $[0,T].$ $L^{2}([0,T];\mathcal
{B})$ is the space of $\mathcal {B}-$valued $L^{2}-$functions on
$[0,T]$. The corresponding space of $\mathcal {B}-$valued functions
on $[0,\infty)$ is defined similarly.
\section{Preliminaries and main result}

\setcounter{equation}{0}

\label{sec:2}

Define the weight function
\begin{equation}
w(\xi)=
\begin{cases}
e^{-\beta(\xi-\xi_{0})},\   \  \ \ \ \xi\leq\xi_{0},\\
1,\ \ \ \ \ \ \ \ \ \ \ \ \ \ \ \xi>\xi_{0},
\end{cases}
\label{201}
\end{equation}
where $\beta$ and $\xi_{0}$ are given in \eqref{104} and \eqref{103}, respectively.

Next, we will state our stability theorem.
\begin{theorem}[Stability]
\label{thm201}  Assume (A1)-(A2) hold. For any given traveling wave fronts $\Phi(x+ct)=(\phi_{1}(x+ct),\phi_{2}(x+ct))$ with a speed
\begin{equation*}
c>\max\Big\{c_{*},\frac{\gamma_{1}-\gamma_{2}}{\beta},
\frac{2\partial_{1}f(0,0)+2\partial_{2}f(0,0)+\gamma_{2}-\gamma_{1}-\frac{D}{2}+D\int_{\mathbb{R}}J(y)e^{-\beta y}dy}{\beta}\Big\}.
\end{equation*}
If the initial data satisfies
\begin{equation*}
\begin{cases}
u_{1-}\leq u_{10}(s,x)\leq u_{1+},\ \ \ (s,x)\in[-\tau,0]\times\mathbb{R},\\
u_{2-}\leq u_{20}(x)\leq u_{2+},\ \ \ \ \ \ x\in\mathbb{R}
\end{cases}
\end{equation*}
and the initial perturbation is
\begin{equation*}
\begin{cases}
u_{10}(s,x)-\phi_{1}(x+cs)\in C([-\tau,0];H_{w}^{1}(\mathbb{R})),\\
u_{20}(x)-\phi_{2}(x)\in H_{w}^{1}(\mathbb{R})\subset C(\mathbb{R}),
\end{cases}
\end{equation*}
then the solution of the Cauchy problem\eqref{101}-\eqref{102} satisfies
$$u_{i}(t,x)-\phi_{i}(x+cs)\in C([0,\infty);H_{w}^{1}(\mathbb{R})),\ \ i=1,2,$$
$$u_{i-}\leq u_{i}(t,x)\leq u_{i+},\ \ (t,x)\in\mathbb{R}_{+}\times\mathbb{R},\ i=1,2,$$
and
$$||(u_{i}-\phi_{i})(t)||_{H_{w}^{1}(\mathbb{R})}\leq Ce^{-\mu t},\ \ t\geq 0,\ \ i=1,2$$
for any constant $0<\mu<\min\{\mu_{1},\mu_{2}\}$, where $\mu_{1}$ and $\mu_{2}$ are defined in Lemma \ref{lemma2}.

In particular, the solution $(u_{1}(t,x), u_{2}(t,x))$ converges to the traveling wave fronts $(\phi_{1}(x+ct),\phi_{2}(x+ct))$ exponentially in time, namely,
$$\sup_{x\in\mathbb{R}}|u_{i}(t,x)-\phi_{i}(x+ct)|\leq Ce^{-\mu t},\ \ t\geq 0,\ \ i=1,2.$$
\end{theorem}
\section{Proof of stability}

\setcounter{equation}{0}
\label{sec:3}

First of all, before proving the stability, we need to establish the boundedness and comparison principle for the solutions of the Cauchy problem \eqref{101}-\eqref{102}.

\begin{lemma}[Boundedness]
\label{lemma3.1}  Assume (A1)-(A2) hold. Let the initial data satisfy
$$(u_{1-},u_{2-})\leq (u_{10}(s,x),u_{20}(x))\leq (u_{1+},u_{2+}),\ \ (t,x)\in[-\tau,0]\times\mathbb{R},$$
then the solution\ $(u_{1}(t,x),u_{2}(t,x))$\ of the Cauchy problem\ \eqref{101}-\eqref{102}\ satisfies
$$(u_{1-},u_{2-})\leq (u_{1}(t,x),u_{2}(t,x))\leq (u_{1+},u_{2+}),\ \ (t,x)\in\mathbb{R}_{+}\times\mathbb{R}.$$
\end{lemma}
Denote
\begin{equation}
\begin{cases}
u_{10}^{+}(s, x)=\max\{u_{10}(s,x),\phi_{1}(x+cs)\},\ \ (s,x)\in[-\tau,0]\times\mathbb{R}, \\
u_{10}^{-}(s, x)=\min\{u_{10}(s,x),\phi_{1}(x+cs)\},\ \ (s,x)\in[-\tau,0]\times\mathbb{R},\\
u_{20}^{+}(x)=\max\{u_{20}(x),\phi_{2}(x)\},\ \ t=0,\ x\in\mathbb{R},\\
u_{20}^{-}(x)=\min\{u_{20}(x),\phi_{2}(x)\},\ \ t=0,\ x\in\mathbb{R}.
\end{cases}
\label{301}
\end{equation}

Next, we will give the comparison principle for the solutions of the Cauchy problem \eqref{101}-\eqref{102}.
\begin{lemma}[Comparison principle]
\label{lemma3.2} Assume (A1)-(A2) hold. Let
$$(u_{1}^{+}(t,x),u_{2}^{+}(t,x))\ and\ (u_{1}^{-}(t,x),u_{2}^{-}(t,x))$$
be the solutions of the Cauchy problem \eqref{101}-\eqref{102} satisfies
$$(u_{10}^{+}(s,x),u_{20}^{+}(x))\ and\ (u_{10}^{-}(s,x),u_{20}^{-}(x))$$
respectively, if
$$(u_{1-},u_{2-})\leq (u_{10}^{-}(s,x),u_{20}^{-}(x))\leq (u_{10}^{+}(s,x),u_{20}^{+}(x))\leq (u_{1+},u_{2+})$$
holds for\ $(s,x)\in[-\tau,0]\times\mathbb{R}$, then we have
$$(u_{1-},u_{2-})\leq (u_{1}^{-}(t,x),u_{2}^{-}(t,x))\leq (u_{1}^{+}(t,x),u_{2}^{+}(t,x))\leq (u_{1+},u_{2+}),\ \ (t,x)\in\mathbb{R}_{+}\times\mathbb{R}.$$
\end{lemma}

The proofs of Lemmas \ref{lemma3.1} and \ref{lemma3.2} are similar to \cite{Guo2018}, so they are omitted for simplicity.

Below we will prove the stability of traveling wave fronts of the system \eqref{101} in two steps.
First, we prove that $u_{i}^{+}(t,x)$ converges to $\phi_{i}(\xi)$.

\begin{proposition}
\label{proposition3.3} Assume (A1)-(A2) hold. Then for any constant $0<\mu<\min\{\mu_{1},\mu_{2}\}$, there is
\begin{eqnarray}
\sup_{x\in\mathbb{R}}|u_{i}^{+}(t,x)-\phi_{i}(x+ct)|\leq Ce^{-\mu t},\ \ i=1,2,
\label{302}
\end{eqnarray}
where $\mu_1$ and $\mu_2$ are defined in Lemma 4.2.
\end{proposition}

\begin{proof}
Let
\begin{equation*}
\begin{cases}
v_{i}(t,\xi)=u_{i}^{+}(t,x)-\phi_{i}(x+ct),\ \ i=1,2,\\
v_{10}(s,\xi)=u_{10}^{+}(s,x)-\phi_{1}(x+cs),\ \ (s,x)\in[-\tau,0]\times \mathbb{R},\\
v_{20}(\xi)=u_{20}^{+}(x)-\phi_{2}(x),\ \ t=0, x\in\mathbb{R}.
\end{cases}
\end{equation*}
Then, by \eqref{301} and Lemma \ref{lemma3.2}, we have
$$v_{i}(t,\xi)\geq 0,\ \ v_{10}(s,\xi)\geq 0,\ \ v_{20}(\xi)\geq 0,\ i=1,2.$$
Meanwhile,\ $v_{i}(t,\xi)$\ satisfies
\begin{equation}
v_{1t}(t,\xi)+cv_{1\xi}(t,\xi)-DJ\ast v_{1}(t,\xi)+(D+\gamma_{1})v_{1}(t,\xi)=P(t,\xi)+\gamma_{2}v_{2}(t,\xi),
\label{303}
\end{equation}
\begin{equation}
v_{2t}(t,\xi)+cv_{2\xi}(t,\xi)+\gamma_{2}v_{2}(t,\xi)=\gamma_{1}v_{1}(t,\xi)
\label{304}
\end{equation}
with the initial data
$$v_{1}(s,\xi)=v_{10}(s,\xi),\ \ (s,\xi)\in[-\tau,0]\times\mathbb{R} \ \ and\ \ v_{2}(\xi)=v_{20}(\xi),\ \xi\in\mathbb{R},$$
where
\begin{eqnarray*}
P(t,\xi)
=f\big(\phi_{1}(\xi)+v_{1}(t,\xi),
\phi_{1}(\xi-c\tau)+v_{1}(t-\tau,\xi-c\tau)\big)
-f\big(\phi_{1}(\xi),\phi_{1}(\xi-c\tau)\big).
\end{eqnarray*}
Furthermore, let us linearize the nonlinear term $P(t,\xi)$ of \eqref{304}, we equivalently obtain
\begin{eqnarray}
&&v_{1t}(t,\xi)+cv_{1\xi}(t,\xi)-DJ\ast v_{1}(t,\xi)+(D+\gamma_{1}-\partial_{1}f\big(\phi_{1}(\xi),\phi_{1}(\xi-c\tau))
v_{1}(t,\xi)\nonumber\\
&=&Q(t,\xi)+\gamma_{2}v_{2}(t,\xi)+\partial_{2}f\big(\phi_{1}(\xi), \phi_{1}(\xi-c\tau)\big)v_{1}(t-\tau,\xi-c\tau),
\label{305}
\end{eqnarray}
where
\begin{eqnarray*}
Q(t,\xi)&=&f\big(\phi_{1}(\xi)+v_{1}(t,\xi),
\phi_{1}(\xi-c\tau)+v_{1}(t-\tau,\xi-c\tau)\big)-f\big(\phi_{1}(\xi),\phi_{1}(\xi-c\tau)\big)\nonumber\\
&&
-\partial_{1}f\big(\phi_{1}(\xi),\phi_{1}(\xi-c\tau)\big)v_{1}(t,\xi)
-\partial_{2}f\big(\phi_{1}(\xi),\phi_{1}(\xi-c\tau)\big)v_{1}(t-\tau,\xi-c\tau)\nonumber\\
&\leq& 0.
\end{eqnarray*}

Multiplying \eqref{305} and \eqref{304} by\ $e^{2\mu t}w(\xi)v_{1}(t,\xi)$ and\ $e^{2\mu t}w(\xi)v_{2}(t,\xi)$,\ respectively, we get
\begin{eqnarray}
&&\Big\{\frac{1}{2}e^{2\mu t}wv_{1}^{2}\Big\}_{t}+\Big\{\frac{c}{2}e^{2\mu t}wv_{1}^{2}\Big\}_{\xi}
+\Big\{-\frac{c}{2}\frac{w'}{w}-\mu+D+\gamma_{1}-\partial_{1}f\big(\phi_{1}(\xi),\phi_{1}(\xi-c\tau))\Big\}
\nonumber\\
&&\cdot e^{2\mu t}wv_{1}^{2}-De^{2\mu t}wv_{1}\int_{\mathbb{R}}J(y)v_{1}(t,\xi-y)dy\nonumber\\
&=&\gamma_{2}e^{2\mu t}wv_{1}v_{2}+e^{2\mu t}wv_{1}Q(t,\xi)+e^{2\mu t}wv_{1}\partial_{2}f\big(\phi_{1}(\xi),\phi_{1}(\xi-c\tau)\big)v_{1}(t-\tau,\xi-c\tau)\nonumber\\
\label{306}
\end{eqnarray}
and
\begin{eqnarray}
\left\{\frac{1}{2}e^{2\mu
t}wv_{2}^{2}\right\}_{t}+\left\{\frac{1}{2}e^{2\mu t}cwv_{2}^{2}\right\}_{\xi}
+\left\{-\frac{c}{2}(\frac{w'}{w})+\gamma_{2}-\mu\right\}
e^{2\mu t}wv_{2}^{2}
=\gamma_{1} e^{2\mu t}wv_{1}v_{2}.\nonumber\\
\label{307}
\end{eqnarray}
Integrating \eqref{306}\ with respect to\ $t$\ and\ $\xi$\ over\ $[0,t]\times\mathbb{R}$, we get
\begin{eqnarray}
&&e^{2\mu t}||v_{1}(t)||_{L_{w}^{2}}^{2}+\int_{0}^{t}\int_{\mathbb{R}}
\Big\{-c\frac{w'}{w}-2\mu+2D+2\gamma_{1}-2\partial_{1}f\big(\phi_{1}(\xi),\phi_{1}(\xi-c\tau))\Big\}\nonumber\\
&&\cdot e^{2\mu s}wv_{1}^{2}d\xi ds-2D\int_{0}^{t}\int_{\mathbb{R}}e^{2\mu s}wv_{1}\int_{\mathbb{R}}J(y)v_{1}(s,\xi-y)dyd\xi ds\nonumber\\
&=&||v_{10}(0)||_{L_{w}^{2}}^{2}+2\int_{0}^{t}\int_{\mathbb{R}}\gamma_{2}e^{2\mu s}wv_{1}v_{2}d\xi ds+
2\int_{0}^{t}\int_{\mathbb{R}}e^{2\mu s}wv_{1}Q(t,\xi)d\xi ds\nonumber\\
&&+2\int_{0}^{t}\int_{\mathbb{R}}e^{2\mu s}wv_{1}\partial_{2}f\big(\phi_{1}(\xi),\phi_{1}(\xi-c\tau)\big)v_{1}(s-\tau,\xi-c\tau)d\xi ds.
\label{308}
\end{eqnarray}
By the Cauchy-Schwarz inequality, we obtain
\begin{eqnarray}
&&2\int_{0}^{t}\int_{\mathbb{R}}e^{2\mu s}wv_{1}\partial_{2}f\big(\phi_{1}(\xi),\phi_{1}(\xi-c\tau)\big)v_{1}(s-\tau,\xi-c\tau)d\xi ds\nonumber\\
&\leq&\int_{0}^{t}\int_{\mathbb{R}}e^{2\mu s}\partial_{2}f\big(\phi_{1}(\xi),\phi_{1}(\xi-c\tau)\big)wv_{1}^{2}d\xi ds\nonumber\\
&&+\int_{0}^{t}\int_{\mathbb{R}}e^{2\mu s}\partial_{2}f\big(\phi_{1}(\xi),\phi_{1}(\xi-c\tau)\big)wv_{1}^{2}(s-\tau,\xi-c\tau)d\xi ds\nonumber\\
&=&\int_{0}^{t}\int_{\mathbb{R}}e^{2\mu s}\partial_{2}f\big(\phi_{1}(\xi),\phi_{1}(\xi-c\tau)\big)wv_{1}^{2}d\xi ds
\nonumber\\
&&+\int_{-\tau}^{t-\tau}\int_{\mathbb{R}}e^{2\mu (s+\tau)}\partial_{2}f\big(\phi_{1}(\xi+c\tau),\phi_{1}(\xi)\big)w(\xi+c\tau)v_{1}^{2}d\xi ds\nonumber\\
&\leq&\int_{0}^{t}\int_{\mathbb{R}}e^{2\mu s}\partial_{2}f\big(\phi_{1}(\xi),\phi_{1}(\xi-c\tau)\big)wv_{1}^{2}d\xi ds
\nonumber\\
&&+e^{2\mu \tau}\int_{-\tau}^{0}\int_{\mathbb{R}}e^{2\mu s}\partial_{2}f\big(\phi_{1}(\xi+c\tau),\phi_{1}(\xi)\big)w(\xi+c\tau)v_{10}^{2}d\xi ds\nonumber\\
&&+e^{2\mu \tau}\int_{0}^{t}\int_{\mathbb{R}}e^{2\mu s}\partial_{2}f\big(\phi_{1}(\xi+c\tau),\phi_{1}(\xi)\big)w(\xi+c\tau)v_{1}^{2}d\xi ds,
\label{309}
\end{eqnarray}
\begin{eqnarray}
\Big|2\int_{0}^{t}\int_{\mathbb{R}}\gamma_{2}e^{2\mu s}wv_{1}v_{2}d\xi ds\Big|
\leq\int_{0}^{t}\int_{\mathbb{R}}\gamma_{2}e^{2\mu s}w[v_{1}^{2}+v_{2}^{2}]d\xi ds
\label{3010}
\end{eqnarray}
and
\begin{eqnarray}
&&2D\int_{0}^{t}\int_{\mathbb{R}}e^{2\mu s}wv_{1}\int_{\mathbb{R}}J(y)v_{1}(s,\xi-y)dyd\xi ds\nonumber\\
&\leq&D\int_{0}^{t}\int_{\mathbb{R}}e^{2\mu s}wv_{1}^{2}d\xi ds+D\int_{0}^{t}\int_{\mathbb{R}}e^{2\mu s}w\int_{\mathbb{R}}J(y)v_{1}^{2}(s,\xi-y)dyd\xi ds\nonumber\\
&=&D\int_{0}^{t}\int_{\mathbb{R}}e^{2\mu s}wv_{1}^{2}d\xi ds+D\int_{0}^{t}\int_{\mathbb{R}}e^{2\mu s}wv_{1}^{2}\int_{\mathbb{R}}J(y)\frac{w(\xi+y)}{w(\xi)}dyd\xi ds.\nonumber\\
\label{3011}
\end{eqnarray}
Thus, \eqref{308} reduces to
\begin{eqnarray}
&&e^{2\mu t}||v_{1}(t)||_{L_{w}^{2}}^{2}+\int_{0}^{t}\int_{\mathbb{R}}
\Big\{-c\frac{w'}{w}-2\mu+D+2\gamma_{1}-2\partial_{1}f\big(\phi_{1}(\xi),\phi_{1}(\xi-c\tau))
\nonumber\\
&&-\partial_{2}f\big(\phi_{1}(\xi),\phi_{1}(\xi-c\tau))-\gamma_{2}-e^{2\mu \tau}\frac{w(\xi+c\tau)}{w(\xi)}\partial_{2}f\big(\phi_{1}(\xi+c\tau),\phi_{1}(\xi)\big)
\nonumber\\
&&-D\int_{\mathbb{R}}J(y)\frac{w(\xi+y)}{w(\xi)}dy\Big\}
e^{2\mu s}wv_{1}^{2}d\xi ds\nonumber\\
&\leq&||v_{10}(0)||_{L_{w}^{2}}^{2}+\int_{0}^{t}\int_{\mathbb{R}}\gamma_{2}e^{2\mu s}wv_{2}^{2}d\xi ds+
2\int_{0}^{t}\int_{\mathbb{R}}e^{2\mu s}wv_{1}Q(t,\xi)d\xi ds\nonumber\\
&&+e^{2\mu \tau}\int_{-\tau}^{0}\int_{\mathbb{R}}e^{2\mu s}\partial_{2}f\big(\phi_{1}(\xi+c\tau),\phi_{1}(\xi)\big)w(\xi+c\tau)v_{10}^{2}d\xi ds.
\label{3012}
\end{eqnarray}
Note that $Q(t,\xi)\leq 0$. Therefore, \eqref{3012} reduces to
\begin{eqnarray}
&&e^{2\mu t}||v_{1}(t)||_{L_{w}^{2}}^{2}+\int_{0}^{t}\int_{\mathbb{R}}
\Big\{-c\frac{w'}{w}-2\mu+D+2\gamma_{1}-2\partial_{1}f\big(\phi_{1}(\xi),\phi_{1}(\xi-c\tau))
\nonumber\\
&&-\partial_{2}f\big(\phi_{1}(\xi),\phi_{1}(\xi-c\tau))-\gamma_{2}-e^{2\mu \tau}\frac{w(\xi+c\tau)}{w(\xi)}\partial_{2}f\big(\phi_{1}(\xi+c\tau),\phi_{1}(\xi)\big)
\nonumber\\
&&-D\int_{\mathbb{R}}J(y)\frac{w(\xi+y)}{w(\xi)}dy\Big\}
e^{2\mu s}wv_{1}^{2}d\xi ds\nonumber\\
&\leq&||v_{10}(0)||_{L_{w}^{2}}^{2}+\int_{0}^{t}\int_{\mathbb{R}}\gamma_{2}e^{2\mu s}wv_{2}^{2}d\xi ds+C\int_{-\tau}^{0}||v_{10}(s)||_{L_{w}^{2}}^{2}ds.
\label{3013}
\end{eqnarray}

Next, integrating\ \eqref{307}\ with respect to\ $t$\ and\ $\xi$\ over\ $[0,t]\times\mathbb{R}$, we have
\begin{eqnarray}
&&e^{2\mu t}\|v_{2}(t)\|_{L_{w}^{2}}^{2}+
\int_{0}^{t}\int_{\mathbb{R}}\left\{-c\frac{w'}{w}
+2\gamma_{2}-2\mu\right\}e^{2\mu
s}w(\xi)v_{2}^{2}(s,\xi)d\xi ds\nonumber\\
&=& ||v_{20}(0)\|_{L_{w}^{2}}^{2}+2\int_{0}^{t}\int_{\mathbb{R}}e^{2\mu
s}\gamma_{1} w(\xi)v_{1}(s,\xi)v_{2}(s,\xi)d\xi ds .
\label{3014}
\end{eqnarray}
Using the Cauchy-Schwarz inequality, we obtain
\begin{eqnarray*}
2\int_{0}^{t}\int_{\mathbb{R}}\gamma_{1} e^{2\mu
s}w(\xi)v_{1}(s,\xi)v_{2}(s,\xi) d\xi ds\leq \int_{0}^{t}\int_{\mathbb{R}}\gamma_{1} e^{2\mu
s}w(\xi)\big[v_{1}^{2}(s,\xi) +v_2^{2}(s,\xi)\big]d\xi ds.
\end{eqnarray*}
Thus, \eqref{3014} is reduced to
\begin{eqnarray}
&&e^{2\mu t}\|v_{2}(t)\|_{L_{w}^{2}}^{2}+
\int_{0}^{t}\int_{\mathbb{R}}\left\{-c\frac{w'}{w}
+2\gamma_{2}-\gamma_{1}-2\mu\right\}e^{2\mu s}w(\xi) v_{2}^{2}(s,\xi)d\xi ds\nonumber\\
&\leq&\|v_{20}(0)\|_{L_{w}^{2}}^{2}
+\int_{0}^{t}\int_{\mathbb{R}}e^{2\mu s}\gamma_{1} w(\xi)v_{1}^{2}(s,\xi)d\xi ds.
\label{3015}
\end{eqnarray}
Combining\ \eqref{3013}\ and\ \eqref{3015}, we obtain
\begin{eqnarray}
&&\sum_{i=1}^{2}e^{2\mu t}\|v_{i}(t)\|_{L_{w}^{2}}^{2}
+\int_{0}^{t}\int_{\mathbb{R}}A_{i}(\mu,\xi)e^{2\mu s}w(\xi)v_{i}^{2}(s,\xi)d\xi ds\nonumber\\
&\leq& \sum_{i=1}^{2}\|v_{i0}(0)\|_{L_{w}^{2}}^{2}
+C\int_{-\tau}^{0}||v_{10}(s)||_{L_{w}^{2}}^{2} ds,\ \ i=1,2,
\label{3016}
\end{eqnarray}
where
\begin{eqnarray*}
A_{1}(\mu,\xi)=B_{1}(\xi)-2\mu
-\frac{w(\xi+c\tau)}{w(\xi)}\partial_{2}f\big(\phi_{1}(\xi+c\tau),\phi_{1}(\xi)\big)(e^{2\mu \tau}-1),
\end{eqnarray*}
$$A_{2}(\mu,\xi)=B_{2}(\xi)-2\mu,$$
and
\begin{eqnarray*}
B_{1}(\xi)&=&-c\frac{w'}{w}+D+\gamma_{1}-2\partial_{1}f\big(\phi_{1}(\xi),\phi_{1}(\xi-c\tau))
-\partial_{2}f\big(\phi_{1}(\xi),\phi_{1}(\xi-c\tau))-\gamma_{2}\nonumber\\
&&-\frac{w(\xi+c\tau)}{w(\xi)}\partial_{2}f\big(\phi_{1}(\xi+c\tau),\phi_{1}(\xi)\big)
-D\int_{\mathbb{R}}J(y)\frac{w(\xi+y)}{w(\xi)}dy,
\end{eqnarray*}
$$B_{2}(\xi)=-c\frac{w'}{w}+\gamma_{2}-\gamma_{1}.$$
According to the inequalities $A_{i}(\mu,\xi)>0 (i=1,2)$ in the appendix, we can get
\begin{eqnarray}
\sum_{i=1}^{2}e^{2\mu t}\|v_{i}(t)\|_{L_{w}^{2}}^{2}\leq C\Big(\sum_{i=1}^{2}\|v_{i0}(0)\|_{L_{w}^{2}}^{2}
+\int_{-\tau}^{0}||v_{10}(s)||_{L_{w}^{2}}^{2} ds\Big),\ i=1,2,
\label{3017}
\end{eqnarray}
which implies
\begin{eqnarray}
e^{2\mu t}\|v_{i}(t)\|_{L_{w}^{2}}^{2}\leq C\Big(\sum_{i=1}^{2}\|v_{i0}(0)\|_{L_{w}^{2}}^{2}
+\int_{-\tau}^{0}||v_{10}(s)||_{L_{w}^{2}}^{2} ds\Big),\ \ i=1,2.
\label{3018}
\end{eqnarray}

Similarly, differentiating\ \eqref{305}\ and \eqref{304}\ with respect to\ $\xi$, and multiplying them by\ $e^{2\mu t}wv_{1\xi}$\ and\ $e^{2\mu t}wv_{2\xi}$,\ respectively. And integrating the resultant equations with respect to\ $t$\ and\ $\xi$\ over\ $[0,t]\times\mathbb{R}$, then by the same method of \eqref{3017}\ and\ \eqref{3018}, we have
\begin{eqnarray}
e^{2\mu t}\|v_{i\xi}(t)\|_{L_{w}^{2}}^{2}\leq C\Big(\sum_{i=1}^{2}\|v_{i0}(0)\|_{H_{w}^{1}}^{2}
+\int_{-\tau}^{0}||v_{10}(s)||_{H_{w}^{1}}^{2} ds\Big),\ i=1,2.
\label{3019}
\end{eqnarray}
Combining\ \eqref{3018} and \eqref{3019}, for any constant $0<\mu<\min\{\mu_{1},\mu_{2}\}$, we get
\begin{eqnarray}
\|v_{i}(t)\|_{H_{w}^{1}}^{2}\leq Ce^{-2\mu t}\Big(\sum_{i=1}^{2}\|v_{i0}(0)\|_{H_{w}^{1}}^{2}
+\int_{-\tau}^{0}||v_{10}(s)||_{H_{w}^{1}}^{2} ds\Big),\ i=1,2.
\label{3020}
\end{eqnarray}
By the Sobolev embedding inequality\ $H^{1}_{w}(\mathbb{R})\hookrightarrow H^{1}(\mathbb{R})\hookrightarrow C(\mathbb{R})$, we obtain
\begin{eqnarray*}
\sup_{\xi\in\mathbb{R}}|v_{i}(t,\xi)|\leq Ce^{-\mu t}\Big(\sum_{i=1}^{2}\|v_{i0}(0)\|_{H_{w}^{1}}^{2}
+\int_{-\tau}^{0}||v_{10}(s)||_{H_{w}^{1}}^{2} ds\Big)^{\frac{1}{2}},\ i=1,2,
\end{eqnarray*}
namely,
\begin{eqnarray*}
\sup_{x\in\mathbb{R}}|u_{i}^{+}(t,x)-\phi_{i}(x+ct)|\leq Ce^{-\mu t},\ i=1,2.
\end{eqnarray*}
The proof is complete.
\end{proof}

Second, we prove that $u_{i}^{-}(t,x)$ converges to $\phi(\xi)$.
Denote
\begin{equation*}
\begin{cases}
v_{i}(t,\xi)=\phi_{i}(x+ct)-u_{i}^{-}(t,x),\ \ i=1,2,\\
v_{10}(s,\xi)=\phi_{1}(x+cs)-u_{10}^{-}(s,x),\\
v_{20}(\xi)=\phi_{2}(x)-u_{20}^{-}(x).
\end{cases}
\end{equation*}
Similar to the Proposition \ref{proposition3.3}, we can give the second proposition.
\begin{proposition}
\label{proposition3.4}  Assume (A1)-(A2) hold. Then for any constant $0<\mu<\min\{\mu_{1},\mu_{2}\}$, there is
\begin{eqnarray*}
\sup_{x\in\mathbb{R}}|u_{i}^{-}(t,x)-\phi_{i}(x+ct)|\leq Ce^{-\mu t},\ \ i=1,2,
\end{eqnarray*}
where $\mu_1$ and $\mu_2$ are defined in Lemma 4.2.
\end{proposition}
\textbf{Proof of Theorem \ref{thm201}}
When $(s,x)\in[-\tau,0]\times\mathbb{R}$, the initial data satisfies
$$u_{10}^{-}(s,x)\leq u_{10}(s,x)\leq u_{10}^{+}(s,x)\ \ and\ \ u_{20}^{-}(x)\leq u_{20}(x)\leq u_{20}^{+}(x).$$
By Lemma \ref{lemma3.1} and Lemma\ \ref{lemma3.2}, we can prove the solutions of the Cauchy problem \eqref{101}-\eqref{102} satisfies
$$u_{i}^{-}(t,x)\leq u_{i}(t,x)\leq u_{i}^{+}(t,x),\ \ (t,x)\in\mathbb{R}_{+}\times\mathbb{R},\ i=1,2.$$
According to Proposition \ref{proposition3.3} and Proposition\ \ref{proposition3.4}, we obtain
\begin{eqnarray*}
\sup_{x\in\mathbb{R}}|u_{i}(t,x)-\phi_{i}(x+ct)|\leq Ce^{-\mu t},\ \ i=1,2
\end{eqnarray*}
for any constant $0<\mu<\min\{\mu_{1},\mu_{2}\}$.

This completes the proof of Theorem \ref{thm201}.
$\hfill\blacksquare$

\section{Appendix}

\setcounter{equation}{0}
\label{sec:4}

 The proof of the key inequalities used in Proposition \ref{proposition3.3} is given below.
\begin{lemma}
\label{lemma1} Let $w(\xi)$ be the weight function defined in \eqref{201}, then
$$B_{i}(\xi)\geq C_{i}> 0,\ \ i=1,2.$$
\end{lemma}
\begin{proof}
First, we will discuss the following two cases to prove $B_{1}(\xi)\geq C_{1}> 0$.

\textbf{Case 1.1} $\xi\leq\xi_{0}$.  From \eqref{201}, we have\ $w(\xi)=e^{-\beta(\xi-\xi_{0})}$. Therefore
\begin{eqnarray*}
B_{1}(\xi)&=&-c\frac{w'}{w}+D+\gamma_{1}-2\partial_{1}f\big(\phi_{1}(\xi),\phi_{1}(\xi-c\tau))
-\partial_{2}f\big(\phi_{1}(\xi),\phi_{1}(\xi-c\tau))-\gamma_{2}\nonumber\\
&&-\frac{w(\xi+c\tau)}{w(\xi)}\partial_{2}f\big(\phi_{1}(\xi+c\tau),\phi_{1}(\xi)\big)
-D\int_{\mathbb{R}}J(y)\frac{w(\xi+y)}{w(\xi)}dy\nonumber\\
&\geq& c\beta+D+\gamma_{1}-2\partial_{1}f(0,0)-2\partial_{2}f(0,0)-\gamma_{2}-D\int_{-\infty}^{\xi_{0}-\xi}J(y)e^{-\beta y}dy\nonumber\\
&&-D\int_{\xi_{0}-\xi}^{\infty}J(y)e^{\beta (\xi-\xi_{0})}dy\nonumber\\
&\geq& c\beta+\frac{D}{2}+\gamma_{1}-2\partial_{1}f(0,0)-2\partial_{2}f(0,0)-\gamma_{2}-D\int_{\mathbb{R}}J(y)e^{-\beta y}dy\nonumber\\
&:=&C_{11}>0.
\end{eqnarray*}

\textbf{Case 1.2} $\xi>\xi_{0}$.\  From \eqref{201}, we have\ $w(\xi)=1$,\ thus
\begin{eqnarray*}
B_{1}(\xi)&=&-c\frac{w'}{w}+D+\gamma_{1}-2\partial_{1}f\big(\phi_{1}(\xi),\phi_{1}(\xi-c\tau))
-\partial_{2}f\big(\phi_{1}(\xi),\phi_{1}(\xi-c\tau))-\gamma_{2}\nonumber\\
&&-\frac{w(\xi+c\tau)}{w(\xi)}\partial_{2}f\big(\phi_{1}(\xi+c\tau),\phi_{1}(\xi)\big)-D\int_{\mathbb{R}}J(y)\frac{w(\xi+y)}{w(\xi)}dy\nonumber\\
&\geq& D+\gamma_{1}-\partial_{1}f(K,K)-\partial_{2}f(K,K)-3\gamma_{2}-D\int_{-\infty}^{\xi_{0}-\xi}J(y)e^{-\beta (\xi+y-\xi_{0})}dy\nonumber\\
&&-D\int_{\xi_{0}-\xi}^{\infty}J(y)dy\nonumber\\
&\geq& \frac{D}{2}+\gamma_{1}-\partial_{1}f(K,K)-\partial_{2}f(K,K)-3\gamma_{2}-D\int_{-\infty}^{0}J(y)e^{-\beta y}dy\nonumber\\
&:=&C_{12}>0.
\end{eqnarray*}
Let\ $C_{1}=\min\{C_{11},C_{12}\},$ then we have\ $B_{1}(\xi)\geq C_{1}> 0$.

Second, we will prove that\ $B_{2}(\xi)\geq C_{2}> 0$.

\textbf{Case 2.1} $\xi\leq\xi_{0}$.  From \eqref{201}, we have\ $w(\xi)=e^{-\beta(\xi-\xi_{0})}$, then
\begin{eqnarray*}
B_{2}(\xi)=-c\frac{w'}{w}+\gamma_{2}-\gamma_{1}=c\beta+\gamma_{2}-\gamma_{1}:=C_{21}>0.
\end{eqnarray*}

\textbf{Case 2.2} $\xi>\xi_{0}$.  From \eqref{201}, we have $w(\xi)=1$, thus
\begin{eqnarray*}
B_{2}(\xi)=-c\frac{w'}{w}+\gamma_{2}-\gamma_{1}=\gamma_{2}-\gamma_{1}:=C_{22}>0.
\end{eqnarray*}
Let\ $C_{2}=\min\{C_{21},C_{22}\}$, then we obtain\ $B_{2}(\xi)\geq C_{2}>0$.

This completes the proof of Lemma \ref{lemma1}.
\end{proof}

\begin{lemma}
\label{lemma2} Assume (A1)-(A2) hold. Then for any constant $0<\mu<\min\{\mu_{1},\mu_{2}\}$, there is
$$A_{1}(\mu,\xi)\geq C_{3}> 0,\ \ A_{2}(\mu,\xi)\geq C_{4}> 0,$$
where\ $\mu_{1}>0$\ is the unique solution of the following equation
$$C_{1}-2\mu-(e^{2\mu \tau}-1)\partial_{2}f(0,0)=0,$$
and\ $\mu_{2}=\frac{C_{2}}{2}>0$ is the solution of equation $C_{2}-2\mu=0$.
\end{lemma}
\begin{proof}
Note that\ $\frac{w(\xi+c\tau)}{w(\xi)}\leq 1,$\ then for\ $0<\mu<\mu_{1}$, we have
\begin{eqnarray*}
A_{1}(\mu,\xi)&=&B_{1}(\xi)-2\mu
-\frac{w(\xi+c\tau)}{w(\xi)}\partial_{2}f\big(\phi_{1}(\xi+c\tau),\phi_{1}(\xi)\big)(e^{2\mu \tau}-1)\nonumber\\
&\geq& C_{1}-2\mu-(e^{2\mu \tau}-1)\partial_{2}f(0,0)\nonumber\\
&:=&C_{3}>0.
\end{eqnarray*}
For\ $0<\mu<\mu_{2}=\frac{C_{2}}{2}$, we have
$$A_{2}(\mu,\xi)=B_{2}(\xi)-2\mu\geq C_{2}-2\mu:=C_{4}>0.$$
The proof is completed.
\end{proof}

\vskip 5mm

\section*{Acknowledgement.}
  The authors sincerely thank the anonymous referees for very careful reading and for providing many inspiring and valuable comments and suggestions which led to great improvement in the earlier version of this paper.
This work is partially supported by NSFC Grants (nos. 11671071, 12071065 and 11871140).





\bibliographystyle{elsarticle-num}
\bibliography{<your-bib-database>}



\end{document}